\documentclass[a4paper,11pt]{amsart}
\usepackage[utf8]{inputenc}
\usepackage{amssymb,amsmath,amsthm, mathtools}
\usepackage[margin=1in]{geometry}
\usepackage{dsfont}
\newtheorem{lemma}{Lemma}
\newtheorem{theorem}{Theorem}
\title[On an anisotropic Serrin crierion]{On an anisotropic Serrin criterion for weak solutions of the Navier-Stokes equations.}
\author{Guillaume L\'evy $^1$}
\address{$^{1}$Laboratoire Jacques-Louis Lions, UMR 7598, Université Pierre
et Marie Curie, 75252 Paris Cedex 05, France.}
\email{$^{1}${levy@ljll.math.upmc.fr}}
\date{}
\begin{document}
\begin{abstract}
In this paper, we draw on the ideas of \cite{CheminZhang} to extend the standard Serrin criterion \cite{Serrin} to an anisotropic version thereof.
Because we work on weak solutions instead of strong ones, the functions involved have low regularity.
Our method summarizes in a joint use of a uniqueness lemma in low regularity and the existence of stronger solutions.
The uniqueness part uses duality in a way quite similar to the DiPerna-Lions theory, first developed in \cite{DiPerna-Lions}.
The existence part relies on $L^p$ energy estimates, whose proof may be found in \cite{CheminZhang}, along with an approximation procedure.
\end{abstract}
\maketitle
\section{Presentation of the problem}
The present paper deals with the regularity of the Leray solutions of the incompressible Navier-Stokes equations in dimension three in space.
We recall that these equations are
\begin{equation}
 \left \{
 \begin{array}{l c}
   \partial_t u + \nabla \cdot (u \otimes u) - \Delta u = - \nabla p, \: \: \: \: \: \: \: \: \: \: \: t \geq 0, \:  x \in \mathbb{X}^3, \\
   \text{div }u \equiv 0, \\
  u(0) = u_0.
 \end{array}
  \right.
  \label{RappelNavierStokes}
 \end{equation}
Here, $u = (u^1, u^2, u^3)$ stands for the velocity field of the fluid, $p$ is the pressure and we have set for simplicity the viscosity equal to $1$.
We use the letter $\mathbb{X}$ to denote $\mathbb{R}$ and $\mathbb{T}$ whenever the current claim or proposition applies to both of them.
Let us first recall the existence theorem proved by J. Leray in his celebrated paper \cite{Leray}.

\begin{theorem}[J. Leray, 1934]
 Let us assume that $u_0$ belongs to the energy space $L^2(\mathbb{X}^3)$. 
 Then there exists at least one vector field $u$ in the energy space $L^{\infty}(\mathbb{R}_+, L^2(\mathbb{X}^3)) \cap L^2(\mathbb{R}_+, H^1(\mathbb{X}^3))$ which solves the system (\ref{RappelNavierStokes}) in the weak sense.
 Moreover, the solution $u$ satisfies for all $t \geq 0$ the energy inequality
$$
  \frac 12 \|u(t)\|_{L^2(\mathbb{X}^3)}^2 + \int_0^t \|\nabla u(s)\|_{L^2(\mathbb{X}^3)}^2 ds \leq \frac 12 \|u_0\|_{L^2(\mathbb{X}^3)}^2.
$$
\end{theorem}

Uniqueness of such solutions, however, remains an outstanding open problem to this day.
In his paper from 1961 \cite{Serrin}, J. Serrin proved that, if one assumes that there exists a weak solution which is mildly regular, then it is actually smooth in space and time.
More precisely, J. Serrin proved that if a weak solution $u$ belongs to $L^p(]T_1, T_2[, L^q(D))$ for $T_2 > T_1 > 0$ and some bounded domain $D \Subset \mathbb{X}$ with the restriction $\frac 2p + \frac 3q < 1$, then this weak solution is $\mathcal{C}^{\infty}$ on $]T_1, T_2[ \times D$.
Following his path, many other authors proved results in the same spirit, with different regularity assumptions and/or covering limit cases.
Let us cite for instance 
\cite{BeiraoDaVeiga}, \cite{CaffKohnNiren}, \cite{CheminZhang}, \cite{FabesJonesRiviere}, \cite{FabreLebeau}, \cite{Giga}, \cite{IskauSereginSverak}, \cite{Struwe}, \cite{vonWahl} 
and references therein.

In this paper, we prove two results of the type we mentioned above : the first one is stated in the torus, while the second one is in a spatial domain in the usual Euclidean space.
Thanks to the compactness of the torus, the first result is easier to prove than its local-in-space counterpart.
For this reason, we will use the torus case as a toy model, thus avoiding many technicalities and enlightening the overall strategy of the proof.

In the torus, the theorem writes as follows.

\begin{theorem}
 Let $u$ be a Leray solution of the Navier-Stokes equations set in $\mathbb{R}_+ \times \mathbb{T}^3$ 
$$
  \left\{ 
  \begin{array}{c c}
   \partial_t u + \nabla \cdot (u \otimes u) - \Delta u = - \nabla p \\
   u(0) = u_0
  \end{array}
  \right.
$$
with initial data $u_0$ in $L^2(\mathbb{T}^3)$ and assume that there exists a time interval $]T_1,T_2[$ such that its third component $u^3$ satisfies
$$
  u^3 \in L^2(]T_1,T_2[, W^{2,\frac 32}(\mathbb{T}^3)).
$$
Then $u$ is actually smooth in time and space on $]T_1,T_2[ \times \mathbb{T}^3$ and satisfies the Navier-Stokes equations in the classical strong sense.
\end{theorem}

In a subdomain of the whole space, we need to add a technical assumption on the initial data, namely that it belongs to some particular $L^p$ space with $p < 2$.
Notice that such an assumption is automatically satisfied in the torus, thank to its compactness.

\begin{theorem}
 Let $u$ be a Leray solution of the Navier-Stokes equations set in $\mathbb{R}_+ \times \mathbb{R}^3$ 
$$
  \left\{ 
  \begin{array}{c c}
   \partial_t u + \nabla \cdot (u \otimes u) - \Delta u = - \nabla p \\
   u(0) = u_0
  \end{array}
  \right.
$$
 with initial data $u_0$ in $L^2(\mathbb{R}^3) \cap L^{\frac 32}(\mathbb{R}^3)$ and assume that there exists a time interval $]T_1,T_2[$ and a spatial domain $D \Subset \mathbb{R}^3$ of compact closure such that its third component $u^3$ satisfies
$$
  u^3 \in L^2(]T_1,T_2[, W^{2,\frac 32}(D)).
$$
Then, on $]T_1,T_2[ \times D$, $u$ is actually smooth in time and space and satisfies the Navier-Stokes equations in the classical strong sense.
\end{theorem}

Compared to the classical case, our result may seem weaker, as we require two space derivatives in $L^{\frac 32}$. 
However, the space in which we assume to have $u^3$ is actually at the same scaling that $L^2(]T_1,T_2[, L^{\infty}(D))$ or $L^2(]T_1,T_2[, BMO(D))$, which are more classically found in regularity theorems such as the one of J. Serrin.
In the scaling sense, our assumption is as strong as the usual Serrin criterion.
We demand a bit more in terms of spatial regularity because of the anisotropic nature of the criterion.

\section{Overview of the proof}

Our strategy draws its inspiration from the anisotropic rewriting of the Navier-Stokes system done in \cite{CheminZhang}, though it also bears resemblance to the work of 
\cite{AmbrosioBV}, \cite{AmbrosioCrippa}, \cite{Depauw}, \cite{LeBrisLions}, \cite{Lerner}.
Letting 
$$\Omega := \text{rot }u = (\omega_1,\omega_2,\omega_3), \: \: \: \omega := \omega_3,$$ 
we notice that $\omega$ solves a transport-diffusion equation with $\Omega \cdot \nabla u^3$ as a forcing term.
This equation writes
\begin{equation}
 \left \{
 \begin{array}{c c}
   \partial_t \omega + \nabla \cdot (\omega u) - \Delta \omega = \Omega \cdot \nabla u^3 \\
  \omega(0) = \omega_0,
 \end{array}
  \right.
  \label{Equation_omega}
 \end{equation}
for some $\omega_0$ which we do not specify.
Actually, because we will assume more regularity on $u^3$ than given by the J.Leray theorem on a time interval which does not contain $0$ in its closure, we will focus our attention on a truncated version of $\omega$, for which the initial data is equal to $0$.
For the clarity of the discussion to follow, we drop any mention of the cut-off terms in this section.
In the same vein, we will act as if Lebesgue spaces on $\mathbb{R}^3$ were ordered, which is of course only true on compact subdomains of $\mathbb{R}^3$. 

Viewing Equation (\ref{Equation_omega}) as some abstract PDE problem, we are able to show, by a classical approximation procedure, 
the existence of \textit{some} solution, call it $\tilde{\omega}$, which belongs to what we shall call the energy space associated to $L^{\frac 65}(\mathbb{X}^3)$, namely 
$$
L^{\infty}(\mathbb{R}_+, L^{\frac 65}(\mathbb{X}^3)) \cap L^2(\mathbb{R}_+, \dot{W}^{1,\frac 65}(\mathbb{X}^3)).
$$
Thanks to Sobolev embeddings, we have $L^{\frac 65}(\mathbb{X}^3) \hookrightarrow \dot{H}^{-1}(\mathbb{X}^3)$ and $\dot{W}^{1,\frac 65}(\mathbb{X}^3) \hookrightarrow L^2(\mathbb{X}^3)$.
In particular, this energy space is a subspace of $L^2(\mathbb{R}_+ \times \mathbb{X}^3)$.
We then conclude than $\tilde{\omega}$ is actually equal to $\omega$ thanks to a uniqueness result in $L^2(\mathbb{R}_+ \times \mathbb{X}^3)$ for Equation (\ref{Equation_omega}).
In particular, our $\omega$ has now an improved regularity, a fact which we will prove useful in the sequel.

At this stage, two things are to be emphasized. 
The first one is that the uniqueness result comes alone, without any existential counterpart.
To put it plainly, we are \emph{not} able to prove existence of solutions in the class where we are seeking uniqueness, contrary to, for instance, the now classical results from DiPerna-Lions \emph{et al}.
The existence here is given from the outside by the very properties of the Navier-Stokes equations.

The second one is the absence of any $L^p$ bound uniform in time in the uniqueness class.
From the algebra of the equation and the regularity assumption we made, one could indeed deduce boundedness in time but only in a Sobolev space of strongly negative index, like $H^{-2}(\mathbb{X}^3)$.
The author is unaware of any uniqueness result for similar equations in such low-regularity spaces of distributions.

We then proceed to decompose the full vorticity $\Omega$ only in terms of $\omega$ and $\partial_3 u^3$, thanks to the div-curl decomposition, otherwise known as the Biot-Savart law.
This decomposition essentially relies on the fact that a $2D$ vector field is determined by its $2D$ vorticity and divergence.
In the case of $(u^1, u^2)$, its $2D$ divergence is $- \partial_3 u^3$, because $u$ is divergence free and its $2D$ vorticity is exactly $\omega$.

Let us introduce some piece of notation, which is taken from \cite{CheminZhang}.
We denote
$$\nabla_h := (\partial_1, \partial_2) \: , \: \nabla_h^{\perp} := (- \partial_2, \partial_1) \: , \: \Delta_h := \partial_1^2 + \partial_2^2. $$

Hence, we can write, denoting $u^h := (u^1, u^2)$,
$$
 u^h = u^h_{\text{curl}} + u^h_{\text{div}},
$$
where 
$$u^h_{\text{curl}} := \nabla_h^{\perp} \Delta_h^{-1}\omega \: , \: u^h_{\text{div}} := \nabla_h \Delta_h^{-1}(- \partial_3u^3). $$

We thus obtain a decomposition of the force $\Omega \cdot \nabla u^3$ into a sum a terms which are of two types.
The first are linear in both $\omega$ and $u^3$, while the others are quadratic in $u^3$ and contain no occurrence of $\omega$.
The first ones write as
$$
\omega \partial_3 u^3 + \partial_2u^3 \partial_3 u^1_{\text{curl}} - \partial_1 u^3 \partial_3 u^2_{\text{curl}},
$$
while the terms quadratic in $u^3$ are
$$
\partial_2 u^3 \partial_3 u^1_{\text{div}} - \partial_1 u^3 \partial_3 u^2_{\text{div}}.
$$
In other words, our $\omega$ is now the solution of some modified, anisotropic transport-diffusion equation with forcing terms.
The forcing terms are exactly those quadratic in $u^3$ mentioned above and by our assumption on $u^3$, they lie in $L^1(\mathbb{R}_+, L^{\frac 32}(\mathbb{X}^3))$.

We use again our strategy based on uniqueness. 
On this new, anisotropic equation, we prove a uniqueness result \emph{in a regularity class in which $\omega$ now lies}, that is, in
$$
L^{\infty}(\mathbb{R}_+, L^{\frac 65}(\mathbb{X}^3)) \cap L^2(\mathbb{R}_+,\dot{W}^{1,\frac 65}(\mathbb{X}^3)),
$$
which is a space of functions more regular than the mere $L^2(\mathbb{R}_+ \times \mathbb{R}^3)$ given by J. Leray existence theorem.
We then proceed to prove the existence of a solution to this anisotropic equation in the energy space associated to $L^{\frac 32}(\mathbb{R}^3)$, which is
$$
L^{\infty}(\mathbb{R}_+, L^{\frac 32}(\mathbb{X}^3)) \cap L^2(\mathbb{R}_+, W^{1,\frac 32}(\mathbb{X}^3)).
$$
Again, Sobolev and Lebesgue embeddings (see the remark in the beginning of this section) entail that the energy space associated to $L^{\frac 32}(\mathbb{X}^3)$ 
embeds in that associated to $L^{\frac 65}(\mathbb{X}^3)$.
Thanks to the second uniqueness result, we deduce once again that $\omega$ has more regularity than assumed.
More precisely, we have proved that $\omega$ lies in 
$$
L^{\infty}(\mathbb{R}_+, L^{\frac 32}(\mathbb{X}^3)) \cap L^2(\mathbb{R}_+, W^{1,\frac 32}(\mathbb{X}^3)).
$$

Now that we have lifted the regularity of $\omega = \omega_3$ to that of $\nabla u^3$, it remains to improve the two other components of the vorticity.
Keeping in mind that we now control two independant quantities in a high regularity space instead of one as we originally assumed, the remainder of the proof shall be easier than its beginning.

At first sight, $\omega_1$ and $\omega_2$ solve two equations which both look very similar to Equation (\ref{Equation_omega}).
Indeed, we have
\begin{equation}
 \left \{
 \begin{array}{c c}
   \partial_t \omega_1 + \nabla \cdot (\omega_1 u) - \Delta \omega_1 = \Omega \cdot \nabla u^1 \\
   \partial_t \omega_2 + \nabla \cdot (\omega_2 u) - \Delta \omega_2 = \Omega \cdot \nabla u^2.
 \end{array}
  \right.
  \label{Systeme_omega}
 \end{equation}
We again make use of the div-curl decomposition, but in a somewhat \emph{adaptative} manner.
Recall that, when we improved the regularity of $\omega_3$, we performed a div-curl decomposition with respect to the third variable.
Such a decomposition has the drawback of forcing the appearance of anisotropic operators, which make lose regularity in some variables and gain regularity in others.

Let us pause for a moment to notice something interesting.
From the div-curl decomposition with respect to the third variable, we know that me way write
$$ 
u^h := (u^1, u^2) = \nabla_h^{\perp} \Delta_h^{-1}\omega + \nabla_h \Delta_h^{-1}(- \partial_3u^3).
$$
Taking the horizontal gradient then gives
$$
\nabla_h u^h = \nabla_h \nabla_h^{\perp} \Delta_h^{-1} \omega + \nabla_h^2 \Delta_h^{-1} (- \partial_3 u^3).
$$
That is, $\nabla_h u^h$ may be written as a linear combination of zero order \emph{isotropic} differential operators applied to $\omega = \omega_3$ and $\partial_3 u^3$.
In other words, as a consequence of the H\"ormander-Mikhlin theorem in three dimensions, the four components of the jacobian matrix $\partial_i u^j, 1 \leq i,j \leq 2$ have the same regularity as $\omega_3$ and $\partial_3 u^3$.

Now that we have some regularity on both $u^3$ and $\omega_3$, we may choose to perform the div-curl decomposition with respect to the second variable for $u^1$ and to the first variable for $u^2$.
Since the $2D$ divergence of $(u^3, u^1)$ is $-\partial_2 u^2$ and its $2D$ vorticity is $\omega_2$, we have 
$$
u^1 = \partial_3 \Delta^{-1}_{(1,3)} \omega_2 - \partial_1 \Delta^{-1}_{(1,3)} \partial_2 u^2.
$$
In turn, taking the derivative with respect to the third variable gives
$$
\partial_3 u^1 = \partial_3^2 \Delta^{-1}_{(1,3)} \omega_2 - \partial_3 \partial_1 \Delta^{-1}_{(1,3)} \partial_2 u^2.
$$
That is, $\partial_3 u^1$ may be expressed as the sum of a term linear in $\omega_2$ and a source term which is, for instance, in $L^2(\mathbb{R}_+, L^3(\mathbb{X}^3))$.
A similar decomposition also applies to $\partial_3 u^2$.
Consequently, the system on $(\omega_1, \omega_2)$ may be recast informally in the following form.
$$
 \left \{
 \begin{array}{c c}
   \partial_t \omega_1 + \nabla \cdot (\omega_1 u) - \Delta \omega_1 = (\text{lin. term in }\omega_2) + (\text{source terms in }L^1(\mathbb{R}_+, L^{\frac 32}(\mathbb{X}^3))) \\
   \partial_t \omega_2 + \nabla \cdot (\omega_2 u) - \Delta \omega_2 = (\text{lin. term in }\omega_1) + (\text{source terms in }L^1(\mathbb{R}_+, L^{\frac 32}(\mathbb{X}^3))).
 \end{array}
  \right.
  %\label{Systeme_omega_isotropise}
$$
Thus, it only remains to prove a uniqueness lemma similar to what we did for Equation (\ref{Equation_omega}), 
along with an existence statement in the energy space associated to $L^{\frac 32}(\mathbb{X}^3)$.
We will then have proved that the full vorticity $\Omega$ was actually in, say, $L^4(\mathbb{R}_+, L^2(\mathbb{X}^3)))$, 
entailing that the whole velocity field lies in $L^1(\mathbb{R}_+, \dot{H}^1(\mathbb{X}^3)))$.
A direct application of the standard Serrin criterion concludes the proof.

\section{Notations}

We define here the notations we shall use in this paper, along with some useful shorthands which we shall make a great use thereof.

If $a$ is a real number or a scalar function, we define for $p > 0$ the generalized power $a^p$ by 
$$
 a^p := a |a|^{p-1}
$$
if $a \neq 0$ and $0$ otherwise.
Such a definition has the advantage of being reversible, in that we have the equality $a = (a^p)^{\frac 1p}$.

Spaces like $L^p(\mathbb{R}_t, L^q(\mathbb{X}_x^3))$ or $L^p(\mathbb{R}_t, W^{s,q}(\mathbb{X}_x^3))$ will have their name shortened simply to $L^pL^q$ and $L^qW^{s,q}$.

As we will have to deal with anisotropy, spaces such as $L^p(\mathbb{R}_t, L^q(\mathbb{X}_z, L^r(\mathbb{X}_{x,y}^2)))$ shall be simply written $L^pL^qL^r$ when the context prevents any ambiguity.

When dealing with regularizations procedures, often done through convolutions, we will denote the smoothing parameter by $\delta$ and the mollifying kernels by $(\rho_{\delta})_{\delta}$.

If $X$ is either a vector or scalar field which we want to regularize, we denote by $X^{\delta}$ the convolution $\rho_{\delta} \ast X$.

Conversely, assume that we have some scalar or vector field $Y$ which is a solution of some (partial) differential equation whose coefficients are generically denoted by $X$.
Both $X$ and $Y$ are to be thought as having low regularity. 
We denote by $Y_{\delta}$ the unique smooth solution of the same (partial) differential equation where all the coefficients $X$ are replaced by their regularized counterparts $X^{\delta}$.

If $1 \leq k \leq n$, the horizontal variable associated to the vertical variable $k$ in $\mathbb{R}^n$ is the $n-1$ tuple of variables $(1, \dots, k-1, k+1, \dots, n)$.
In practice, we will restrict our attention to $n=3$, in which case the horizontal variable associated to, say, $3$ is none other than $(1,2)$.

Now, for $1 \leq i,j,k \leq 3$, we denote by $A^k_{i,j}$ the operator $\partial_i \partial_j \Delta_{h_k}^{-1}$, with $h_k$ being the horizontal variable associated to the vertical variable $k$.
We divide these $18$ operators into three subsets.

First, we say that $A^k_{i,j}$ is \textit{isotropic} if we have both $i \neq k$ and $j \neq k$. 
This corresponds to the case where the two derivatives lost through the derivations are actually gained by the inverse laplacian. 
Applying the H\"ormander-Mikhlin multiplier theorem in two dimensions shows that these operators are bounded from $L^p(\mathbb{X}^3)$ to itself for any $1 < p < \infty$.
There are $9$ such operators.

The second class is that of the $A^k_{i,j}$ for which exactly one on the two indices $i$ and $j$ is equal to $k$ while the other is not. 
We say that such operators are \textit{weakly anisotropic}.
Here, we lose one derivative in the vertical variable and gain one in the horizontal variable.
There are $6$ such operators.

The third and last class, which we will not have to deal with in this paper thanks to the peculiar algebraic structure of the equations, is formed by the three $A^k_{k,k} = \partial_k^2 \Delta_{h_k}^{-1}$ for $1 \leq k \leq 3$.
To keep a consistent terminology, we call them \textit{strongly anisotropic}. 
The fact that we lose two derivatives in the vertical variable and gain two derivatives in the horizontal variable while working in dimension $3$ should make this last family quite nontrivial to study.

If $A$ and $B$ are two linear operators, their commutator is defined by $[A,B] := AB - BA$.
We emphasize that, when dealing with commutators, we do not distinguish between a smooth function and the multiplication operator by the said function.

\section{Preliminary lemmas}

We collect in this section various results, sometimes taken from other papers which we will use while proving the main theorems.
We begin by an analogue of the usual energy estimate,whose proof may be found in \cite{CheminZhang} except it is performed in $L^p$ with $p \neq 2$.
\begin{lemma}
 Let $1 < p < \infty$ and $a_0$ in $L^p$. Let $f$ be in $L^1L^p$ and $v$ be a divergence-free vector field in $L^2L^{\infty}$.
 Assume that $a$ is a smooth solution of
 $$
 \left \{
 \begin{array}{c c}
   \partial_t a + \nabla \cdot (a \otimes v) - \Delta a = f \\
  a(0) = a_0.
 \end{array}
  \right.
$$
 Then, $|a|^{\frac p2}$ belongs to $L^{\infty}L^2 \cap L^2H^1$ and we have the $L^p$ energy equality 
 \begin{multline*}
 \frac 1p \|a(t)\|_{L^p}^p + (p-1) \int_0^t \||a(s)|^{\frac{p-2}{2}} \nabla a(s)\|_{L^2}^2 ds \\
   = \frac 1p \|a_0\|_{L^p}^p + \int_0^t \int_{\mathbb{R}^3} f(s,x) a(s,x) |a(s,x)|^{p-2} dx ds.
\end{multline*}
\label{EstimeeEnergieLp}
\end{lemma}

Our next lemma is, along with the energy estimate above, one of the cornerstones of our paper.
Thanks to it, we are able to prove that the solutions of some PDEs are more regular than expected.
It may be found in \cite{NoteAuCRAS} and appear as a particular case of Theorem $2$ in \cite{CRASrefait}, to which we refer the reader for a detailed proof.

\begin{lemma}
 Let $v$ be a fixed, divergence free vector field in $L^2H^1$.
 Let $\nu \geq 0$ be a real constant.
 Let $a$ be a $L^2_{loc}L^2$ solution of 
$$
 \left \{
 \begin{array}{c c}
   \partial_t a + \nabla \cdot (a \otimes v) - \nu \Delta a = 0 \\
  a(0) = 0.
 \end{array}
  \right.
$$
Then $a \equiv 0$.
\label{UniciteL2L2}
\end{lemma}

The following lemma has a somewhat probabilistic flavor to it.

\begin{lemma}
 Let $(a_{\delta})_{\delta}$ be a sequence of bounded functions in $L^pL^q$, with  $1 \leq  p,q \leq \infty$. Let $a$ be in $L^pL^q$ and assume that
$$
 \left \{
 \begin{array}{l l}
  a_{\delta} \rightharpoonup^* a \text{ in } L^pL^q  \\
  a_{\delta} \to a \text{ a.e.}
 \end{array}
  \right.
$$
 as $\delta$ goes to $0$.
 %If $p = \infty$, the first assumption is replaced by 
 %\begin{equation}
 % a_{\delta} \rightharpoonup^{\ast} a \text{ in } L^{\infty}L^q.
 %\end{equation}

 Then, for any $\alpha \in ]0,1[$, 
$
  a_{\delta}^{\alpha} \rightharpoonup^* a^{\alpha}\text{ in }L^\frac{p}{\alpha} L^\frac{q}{\alpha}.
$
 %if $p$ is finite. Else, we have
 %\begin{equation}
 % a_{\delta}^{\alpha} \rightharpoonup^{\ast} a^{\alpha}\text{ in }L^{\infty} L^\frac{q}{\alpha}.
 %\end{equation}
 \label{ConvergenceFaiblePuissances}
\end{lemma}

\begin{proof}
 Let us fix some $\alpha$ in $]0,1[$ and let $p' := (1 - \frac{\alpha}{p})^{-1}$ , $q' := (1 - \frac{\alpha}{q})^{-1}$.
 Let $g$ be a smooth function with compact support in space, which we denote by $S$. 
 Let us remark that, from the assumptions we made, $a_{\delta}^{\alpha} \to a^{\alpha}$ almost everywhere.
 By Egorov's theorem, because $[0,T] \times S$ has finite Lebesgue measure, for any $\varepsilon > 0$, 
 there exists a subset $A_{\varepsilon}$ of $[0,T] \times S$ of Lebesgue measure at most $\varepsilon$ such that
$$
  \|a_{\delta}^{\alpha} - a^{\alpha}\|_{L^{\infty}(A_{\varepsilon}^c)} \to 0 \text{ as } \delta \to 0,
$$
 where we use $A_{\varepsilon}^c$ as a shorthand for $([0,T] \times S) \setminus A_{\varepsilon}$.
Out of the bad set $A_{\varepsilon}$, we can simply write
 \begin{align*}
  \left| \int_0^T \int_S (a_{\delta}^{\alpha} - a^{\alpha}) g \mathds{1}_{A_{\varepsilon}^c} dx dt \right| \leq  
  \|a_{\delta}^{\alpha} - a^{\alpha}\|_{L^{\infty}(A_{\varepsilon}^c)} \|g\|_{L^1L^1},
 \end{align*}
and this last quantity goes to $0$ as $\delta$ goes to $0$, for any fixed $\varepsilon$.
Let $\mu_{\varepsilon}(t) := \int_S \mathds{1}_{\varepsilon}(t,x) dx$. 
We notice that $\|\mu\|_{L^1} \leq \varepsilon$, while $\|\mu\|_{L^{\infty}} \leq C$ for some $C$ independant of $\varepsilon$.
By interpolation, this gives $\|\mu\|_{L^{p'}} \lesssim \varepsilon^{\frac{1}{p'}}$.
On $A_{\varepsilon}$, we have
\begin{align*}
 \left| \int_0^T \int_S a_{\delta}^{\alpha} \mathds{1}_{A_{\varepsilon}} g dx dt \right| 
 & \leq \int_0^T \|a_{\delta}^{\alpha}\|_{L^{\frac{q}{\alpha}}(S)} \|g\|_{L^{\infty}} \mu_{\varepsilon}^{\frac{1}{q'}} dt \\
 & \leq \|a_{\delta}^{\alpha}\|_{L^{\frac{p}{\alpha}}L^{\frac{q}{\alpha}}} \|g\|_{L^{\infty}L^{\infty}} \|\mu\|_{L^{\frac{p'}{q'}}}^{\frac{1}{q'}} \\
 & \lesssim \varepsilon^{\frac{1}{p'}}.
\end{align*}
Similarly,
$$
 \left| \int_0^T \int_S a^{\alpha} \mathds{1}_{A_{\varepsilon}} g dx dt \right| \lesssim \varepsilon^{\frac{1}{p'}}.
$$
Letting first $\delta$ then $\varepsilon$ go to $0$, thanks to the fact that $p'$ is finite, we get the desired convergence.
 The case of a general $g$ in $L^{p'}L^{q'}$ is handled by a standard approximation procedure, which is made possible by the finiteness of both $p'$ and $q'$.
\end{proof}

\begin{lemma}
 Let $F$ be in $L^1L^1$ spatially supported in the ball $B(0,R)$ for some $R > 0$. Let $a$ be the unique tempered distribution solving 
$$
 \left \{
 \begin{array}{c c}
   \partial_t a - \Delta a = F \\
  a(0) = 0.
 \end{array}
  \right.
$$
  Then there exists a constant $C = C_R > 0$ such that, for $|x| > 2R$, we have
  \begin{equation}
   |a(t,x)| \leq C_R \|F\|_{L^1L^1} |x|^{-3}.
  \end{equation}
  \label{DecroissanceChaleur}
\end{lemma}

\begin{proof}
 Let us write explicitly the Duhamel formula for $a$. We have, thanks to the support assumption on $F$,
$$
  a(t,x) = \int_0^t \int_{B(0,R)} (2 \pi (t-s))^{-\frac 32} e^{-\frac{|x-y|^2}{4(t-s)}} F(s,y) dy ds.
$$
As the quantity $\tau^{-3/2} e^{-A^2/\tau}$ reaches its maximum for $\tau = \frac{2A^2}{3}$, we have
$$
  |a(t,x)| \lesssim \int_0^t \int_{B(0,R)} |x-y|^{-3} |F(s,y)| dy ds.
$$
 If $x$ lies far away from the support of $F$, for instance if $|x| > 2R$ in our case, we further have
$$
  |a(t,x)| \leq C_R \int_0^t \int_{B(0,R)} |x|^{-3} |F(s,y)| dy ds = C_R |x|^{-3} \|F\|_{L^1L^1}.
$$
\end{proof}

The following lemma is an easy exercise in functional analysis, whose proof will be skipped.

\begin{lemma}
 Let us define, for some fixed $R > 0$ and $p > 1$, the space 
$$
  \tilde{W}^{1,p}(\mathbb{R}^3) := \{ u \in W^{1,p}(\mathbb{R}^3) \text{ s.t. } \sup_{|x| > 2R} |x|^3 |u(x)| < \infty \}.
$$
Then the embedding of $\tilde{W}^{1,p}$ into $L^p$ is compact.
\label{InjectionSobolevCompacte}
\end{lemma}

The next lemma combines some of the previous ones and plays a key role in the paper.
It allows us to gain regularity on the solutions to transport-diffusion equations for free.

\begin{lemma}
 Let $v$ be a fixed, divergence free vector field in $L^2H^1$.
 Let $\frac 65 < p \leq 2$.
 Let $F = (F_i)_i$ be in $L^1L^p$ and assume that $a = (a_i)_i$ is a solution in $L^2L^2$ of 
$$
 \left \{
 \begin{array}{c c}
   \partial_t a + \nabla \cdot (a \otimes v) - \Delta a = F \\
  a(0) = 0.
 \end{array}
  \right.
$$
 
Then $a$ is actually in $L^{\infty}L^p \cap L^2W^{1,p}$ and moreover, its $i$-th component $a_i$ satisfies the energy inequality
$$
 \frac 1p \|a_i(t)\|_{L^p}^p + (p-1) \int_0^t \||a_i(s)|^{\frac{p-2}{2}} \nabla a_i(s)\|_{L^2}^2 ds
 \leq \int_0^t \int_{\mathbb{R}^3} a_i^{p-1}(s) F_i(s) dx ds.
$$
\label{GainRegulariteLp}
\end{lemma}

\begin{proof}
Before delving into the proof itself, we begin with a simplifying remark.
As the equation on $a_i$ simply writes
$$
 \partial_t a_i + \nabla \cdot (a_i v) - \Delta a_i = F_i,
$$
the equations on the $a_i$ are uncoupled, which allows us prove to prove the lemma only in the scalar case. 
Thus, we assume in the rest of the proof the $a$ is actually a scalar function.

Let $(\rho_{\delta})_{\delta}$ be a sequence of space-time mollifiers.
Let $a_{\delta}$ be the unique solution of the Cauchy system
$$
 \left \{
 \begin{array}{c c}
   \partial_t a_{\delta} + \nabla \cdot (a_{\delta} v^{\delta}) - \Delta a_{\delta} = F^{\delta} \\
  a_{\delta}(0) = 0.
 \end{array}
  \right.
$$
Performing an energy-type estimate in $L^p$, which is made possible thanks to Lemma \ref{EstimeeEnergieLp}, we get for all strictly positive $t$ the equality
$$
 \frac 1p \|a_{\delta}(t)\|_{L^p}^p + (p-1) \int_0^t \||a_{\delta}(s)|^{\frac{p-2}{2}} \nabla a_{\delta}(s)\|_{L^2}^2 ds
 = \int_0^t \int a_{\delta}^{p-1}(s) F^{\delta}(s) dx ds
$$
In turn, it entails that
$$
 \|a_{\delta}(t)\|_{L^p} \leq p \int_0^t \|F^{\delta}(s)\|_{L^p} ds,
$$
which finally gives
$$
  \frac 1p \|a_{\delta}(t)\|_{L^p}^p + (p-1) \int_0^t \||a_{\delta}(s)|^{\frac{p-2}{2}} \nabla a_{\delta}(s)\|_{L^2}^2 ds
 \leq p^{p-2} \left( \int_0^t \|F^{\delta}(s)\|_{L^p} ds  \right)^p.
$$
From the definition of $F^{\delta}$, we infer that
$$
  \frac 1p \|a_{\delta}(t)\|_{L^p}^p + (p-1) \int_0^t \||a_{\delta}(s)|^{\frac{p-2}{2}} \nabla a_{\delta}(s)\|_{L^2}^2 ds
 \leq p^{p-2} \left( \int_0^t \|F(s)\|_{L^p} ds  \right)^p,
$$
where the last term is independent of $\delta$. 
Because $p < 2$, we have a bound on $a_{\delta}$ in $L^{\infty}L^p \cap L^2W^{1,p}$ uniform in $\delta$, thanks to the identity $\nabla a = (\nabla a |a|^{\frac{p-2}{2}})|a|^{\frac{2-p}{2}}$.

We know take the limit $\delta \to 0$. First of all, because $F^{\delta}$ is nothing but a space-time mollification of $F$, we have
$$
 \|F^{\delta} - F\|_{L^1L^p} \to 0 \text{ as } \delta \to 0.
$$

Moreover, the weak-$\ast$ accumulation points of $(a_{\delta})_{\delta}$ in $L^{\infty}L^p$ and $L^2W^{1,p}$ respectively are, in particular,
solutions of the problem
$$
 \left \{
 \begin{array}{c c}
   \partial_t b + \nabla \cdot (b v) - \Delta b = F \\
  b(0) = 0.
 \end{array}
  \right.
$$
Because $p \geq \frac 65$, the space $W^{1,p}(\mathbb{R}^3)$ embeds into $L^q$ for some $q \geq 2$. 
By Lemma \ref{UniciteL2L2}, the only possible accumulation point is none other than $a$.
Thus, as $\delta \to 0$,
$$
 \left \{
 \begin{array}{l l}
   a_{\delta} \rightharpoonup^{\ast} a \text{ in } L^{\infty}L^p \\
   a_{\delta} \rightharpoonup a \text{ in } L^2W^{1,p}.
 \end{array}
  \right.
$$
 From Lemma \ref{DecroissanceChaleur}, we also have
$$
  |a_{\delta}(t,x)| \lesssim |x|^{-3}
$$
for large enough $x$, with constants independant of $\delta$.
Combining the bounds we have on the family $(a_{\delta})_{\delta}$, we have shown that this family is bounded in $L^2_{loc}\tilde{W}^{1,p}$.
On the other hand, the equation on $a_{\delta}$ may be rewritten as
$$
 \partial_t a_{\delta} = - \nabla \cdot (a_{\delta} \otimes v^{\delta}) + \Delta a_{\delta} + F^{\delta}
$$
and the right-hand side is bounded in, say, $L^1_{loc}H^{-2}$, because $p \geq \frac 65$.
 By Aubin-Lions lemma, it follows that the family $(a_{\delta})_{\delta}$ is strongly compact in, say, $L^2_{loc}L^p$.
 Furthermore, once again thanks to Lemma \ref{UniciteL2L2}, it follows that $a$ is the only strong accumulation point of $(a_{\delta})_{\delta}$ in $L^2_{loc}L^p$.
 Thus, 
$$
  a_{\delta} \to a \text{ in } L^2_{loc}L^p.
$$
Thanks to this strong convergence, up to extracting a subsequence $(\delta_n)_n$, we have
$$
 a_{\delta_n} \to a \text{ a.e. as } n \to \infty.
$$
We are now in position to apply Lemma \ref{ConvergenceFaiblePuissances} to the sequence $(a_{\delta_n})_n$.
With $\alpha = \frac p2$, we have 
$$
 a_{\delta_n}^{\frac p2} \rightharpoonup^* a^{\frac p2} \text{ in } L^{\infty}L^2 \text{ as } n \to \infty,
$$
while $\alpha = p-1$ leads to
$$
 a_{\delta_n}^{p-1} \rightharpoonup^* a^{p-1} \text{ in } L^{\infty}L^{\frac{p}{p-1}} \text{ as } n \to \infty.
$$
Using the identity $\nabla(a^{\frac p2}) = \frac p2 a^{\frac{p-2}{2}} \nabla a$ and the energy inequality, we have
$$
 \sup_{n \in \mathbb{N}} \int_0^t \|\nabla(a_{\delta_n}^{\frac p2})\|_{L^2}^2 ds < \infty. 
$$
Since $a_{\delta_n}^{\frac p2} \rightharpoonup^* a^{\frac p2} \text{ in } L^{\infty}L^2 \text{ as } n \to \infty$, applying Fatou's lemma to $a^{\frac p2}$ shows that
$$
 \int_0^t \|\nabla(a^{\frac p2})\|_{L^2}^2 ds \leq \liminf_{n \to \infty} \int_0^t \|\nabla(a_{\delta_n}^{\frac p2})\|_{L^2}^2 ds < \infty.
$$
Taking the limit in the energy inequality, we finally have
$$
 \frac 1p \|a(t)\|_{L^p}^p + (p-1) \int_0^t \||a(s)|^{\frac{p-2}{2}} \nabla a(s)\|_{L^2}^2 ds
 \leq p^{p-2} \left( \int_0^t \|F(s)\|_{L^p} ds  \right)^p.
$$
More interestingly, taking the limit in the energy equality gives us the stronger statement
$$
 \frac 1p \|a(t)\|_{L^p}^p + (p-1) \int_0^t \||a(s)|^{\frac{p-2}{2}} \nabla a(s)\|_{L^2}^2 ds
 \leq \int_0^t \int_{\mathbb{R}^3} a^{p-1}(s) F(s) dx ds .
$$
The proof of the lemma is now complete.
\end{proof}

\begin{lemma}
 Let $v$ be a fixed, divergence free vector field in $L^2H^1$.
 Let $A$ be a matrix-valued function in $L^2L^3$.
 Let $K$ be a matrix whose coefficients are homogeneous Fourier multipliers of order $0$, smooth outside the origin.
 Let $a$ be a solution in $(L^2L^2)^2$ of the equation
 
$$
  \left\{ 
  \begin{array}{c c}
   \partial_t a + \nabla \cdot (a \otimes v) - \Delta a = A K a \\
   a(0) = 0.
  \end{array}
  \right.
$$
Then $a = 0$.
\label{UniciteL2L2Systeme}
\end{lemma}

\begin{proof}
From the assumptions we made, the right-hand side $AK a$ lies in $L^1L^{\frac 65}$.
Thanks to Lemma \ref{GainRegulariteLp}, $a$ is actually in $L^{\infty} L^{\frac 65} \cap L^2 W^{1, \frac 65}$. 
Moreover, we also have a set of energy estimates in $L^{\frac 65}$ on the components $a_i$ of $a$, which are
$$
 \frac{5}{6} \|a_i(t)\|_{L^{\frac 65}}^{\frac 65} + \frac{1}{5} \int_0^t \||a_i(s)|^{- \frac 25} \nabla a_i(s)\|_{L^2}^2 ds
 \leq \int_0^t \int_{\mathbb{R}^3} a_i^{\frac 15}(s) (A(s) K a(s))_i dx ds.
$$

By Hölder inequality and Sobolev embeddings, we have
\begin{align*}
 \int_0^t \int_{\mathbb{R}^3} a_i^{\frac 15}(s) (A(s) K a(s))_i dx ds 
 & \lesssim  \sum_j \int_0^t \|A(s)\|_{L^3} \|a_i(s)\|_{L^{\frac 65}}^{\frac 15} \|a_j(s)\|_{L^2} ds \\
 & \lesssim  \sum_j \int_0^t \|A(s)\|_{L^3} \|a_i(s)\|_{L^{\frac 65}}^{\frac 15} \|\nabla a_j(s) |a_j(s)|^{-\frac 25}\|_{L^2} \|a_j(s)\|_{L^{\frac 65}}^{\frac 25} ds \\
 & \lesssim  \sum_j \int_0^t \|A(s)\|_{L^3} \|a(s)\|_{L^{\frac 65}}^{\frac 35} \|\nabla a_j(s) |a_j(s)|^{-\frac 25}\|_{L^2} ds.
\end{align*}

Young inequality now ensures that
\begin{multline*}
 \int_0^t \|A(s)\|_{L^3} \|a(s)\|_{L^{\frac 65}}^{\frac 35} \|\nabla a_j(s) |a_j(s)|^{-\frac 25}\|_{L^2} ds \\
 \leq \frac{1}{10} \int_0^t \|\nabla a_j(s) |a_j(s)|^{-\frac 25}\|_{L^2}^2 ds + C \int_0^t \|A(s)\|_{L^3}^2 \|a(s)\|_{L^{\frac 65}}^{\frac 65}  ds.
\end{multline*}

Adding these inequalities and cancelling out the gradient terms, we get
$$
 \frac{5}{6} \|a(t)\|_{L^{\frac 65}}^{\frac 65}
 \lesssim \int_0^t  \|A(s)\|_{L^3}^2 \|a(s)\|_{L^{\frac 65}}^{\frac 65} ds.
$$
Grönwall inequality now implies that $a = 0$.
\end{proof}

\begin{lemma}
 Let $ \frac 65 \leq p \leq 2$.
 Let $v$ be a fixed, divergence free vector field in $L^2H^1$.
 Let $A$ be a matrix-valued function in $L^2L^3$.
 Let $K$ be a matrix whose coefficients are homogeneous Fourier multipliers of order $0$, smooth outside the origin.
 Let $F$ be a fixed function in $L^1L^p$.
 Let $a$ be a solution in $(L^2L^2)^2$ of the equation
 
$$
  \left\{ 
  \begin{array}{c c}
   \partial_t a + \nabla \cdot (a \otimes v) - \Delta a = A K a + F  \\
   a(0) = 0.
  \end{array}
  \right.
$$
Then $a$ is actually in $L^{\infty} L^p \cap L^2W^{1, p}$.
\label{GainRegulariteSysteme}
\end{lemma}

\begin{proof}
The proof follows closely the steps of Lemma \ref{GainRegulariteLp}, so we shall skip it.
\end{proof}

\begin{lemma}
 Let $v$ be a fixed, divergence free vector field in $L^2H^1$.
 Let $a$ be a $L^{\infty}L^{\frac 65} \cap L^2W^{1,\frac 65}$ solution of the linear system
$$
 \left \{
 \begin{array}{c c}
   \partial_t a + \nabla \cdot (a v) - \Delta a = \alpha a + \sum_{i,j = 1,2} \varepsilon_{i,j} (\partial_j \beta_i) A^3_{3,i}  a \\ 
  a(0) = 0,
 \end{array}
  \right.
  \label{EquationComplete}
$$
 with $\varepsilon_{i,j} \in \{0,1\}$ for any $1 \leq i,j \leq 2$. 
 We also assume that $\alpha$ lies in $L^2L^3$ and that all the $\beta_i$'s are in $L^2H^{\frac 32}$.
 Then $a \equiv 0$.
 \label{UniciteAnisotropeL6/5}
\end{lemma}

\begin{proof}
For the sake of readability, we assume in the proof that only one coefficient $\varepsilon_{i,j}$ is not zero. We denote the corresponding $\partial_j \beta_i$ simply by $\partial_j \beta$.
Let us denote by $F$ the right-hand side of (\ref{EquationComplete}).
From the assumptions and anisotropic Sobolev embeddings, it follows that $F$ belongs to $L^1L^{\frac 65}$.
By Lemma \ref{GainRegulariteLp}, $a$ satisfies an energy inequality which writes, in our case,
\begin{multline*}
 \frac{5}{6} \|a(t)\|_{L^{\frac 65}}^{\frac 65} + \frac{1}{5} \int_0^t \||a(s)|^{-\frac 25} \nabla a(s)\|_{L^2}^2 ds \\
   \leq \int_0^t \int_{\mathbb{R}^3} \left( a^{\frac 65}(s) \alpha(s) + a^{\frac 15}(s) \partial_j \beta(s) A^3_{3,i} a(s) \right) dx ds.
\end{multline*}
By H\"older inequalities, we have
\begin{eqnarray*}
 \int_0^t \int_{\mathbb{R}^3}  a^{\frac 65}(s) \alpha(s)dx ds & \lesssim & \int_0^t \|a^{\frac 35}(s)\|_{L^3}^2 \|\alpha(s)\|_{L^3} ds \\
 & \lesssim & \int_0^t \|a^{\frac 35}(s)\|_{L^2} \||a(s)|^{-\frac 25} \nabla a(s)\|_{L^2} \|\alpha(s)\|_{L^3} ds \\
 & \leq & \frac{1}{10} \int_0^t \||a(s)|^{-\frac 25} \nabla a(s)\|_{L^2}^2  ds + C \int_0^t \|a^{\frac 35}(s)\|_{L^2}^2 \|\alpha(s)\|_{L^3}^2 ds.
\end{eqnarray*}
To bound the other term, we begin by using a trace theorem on $\beta$, which gives $\beta \in L^2L^{\infty}H^1$.
Taking a horizontal derivative, we get $\partial_j \beta \in L^2L^{\infty}L^2$.
We emphasize that such a trace embedding would not be true in general, because $H^{\frac 12}(\mathbb{X})$ does not embed in $L^{\infty}(\mathbb{X})$.
Here, the fact that the multiplicator $\partial_j \beta$ appears as a derivative of some function is crucial.
Regarding the weakly anisotropic term $A^3_{3,i}a$, the assumption on $a$ gives $\partial_3 a \in L^2L^{\frac 65} = L^2L^{\frac 65} L^{\frac 65}$.
Since in two dimensions the space $W^{1,\frac 65}$ embeds into $L^3$, we get $A^3_{3,i} a \in L^2L^{\frac 65} L^3$.
Combining these embeddings with H\"older inequality, we arrive at 
\begin{align*}
 \int_0^t \int_{\mathbb{R}^3} a^{\frac 15}(s) \partial_j \beta(s) A^3_{3,i} a(s) dx ds
 & \leq \int_0^t \|a^{\frac 15}(s)\|_{L^6L^6} \|\partial_j \beta(s)\|_{L^{\infty}L^2} \|A^3_{3,i}a(s)\|_{L^{\frac 65}L^3} ds \\
 & \lesssim \int_0^t \|a^{\frac 15}(s)\|_{L^6} \|\beta(s)\|_{H^{\frac 32}} \|\nabla a(s)\|_{L^{\frac 65}}.
\end{align*}
Using the identity $\nabla a = \left( |a|^{-\frac 25}\nabla a \right) |a|^{\frac 25}$ and H\"older inequality again, we get
$$
  \int_0^t \int_{\mathbb{R}^3} a^{\frac 15}(s) \partial_j \beta(s) A^3_{3,i} a(s) dx ds
  \lesssim \int_0^t \|a^{\frac 35}(s)\|_{L^2} \|\beta(s)\|_{H^{\frac 32}} \||a(s)|^{-\frac 25}\nabla a(s)\|_{L^2}.
$$
Now, Young inequality for real numbers entails, for some constant $C$,
$$
 \int_0^t \int_{\mathbb{R}^3} a^{\frac 15}(s) \partial_j \beta(s) A^3_{3,i} a(s) dx ds 
  \leq  \frac{1}{10} \int_0^t \||a(s)|^{-\frac 25} \nabla a(s)\|_{L^2}^2 ds 
  +  C\int_0^t \|a^{\frac 35}(s)\|_{L^2}^2 \|\beta(s)\|_{H^{\frac 32}}^2 ds.
$$
Cancelling out the gradient terms, we finally get
$$
\|a^{\frac 35}(t)\|_{L^2}^2 \lesssim \int_0^t \|a^{\frac 35}(s)\|_{L^2}^2 (\|\alpha(s)\|_{L^3}^2 + \|\beta(s)\|_{H^{\frac 32}}^2) ds.
$$
Gr\"onwall's inequality then ensures that $\|a^{\frac 35}(t)\|_{L^2}^2 \equiv 0$ and thus that $a \equiv 0$.
\end{proof}

The three following lemmas allow us, in the spirit of of Lemmas \ref{GainRegulariteLp} and \ref{UniciteAnisotropeL6/5}, to enhance the regularity of the solutions to some equations.
As their proofs are akin to those of the aforementioned Lemmad we only sketch them.

\begin{lemma}
Let $\frac 65 \leq p \leq 2$.
Let $v$ be a fixed, divergence free vector field in $L^2H^1$.
Let $a$ be a solution in $L^{\infty}L^{\frac 65} \cap L^2W^{1,\frac 65}$ of the linear system
$$
 \left \{
 \begin{array}{c c}
   \partial_t a + \nabla \cdot (a v) - \Delta a = \alpha a + \sum_{i,j = 1,2} \varepsilon_{i,j} (\partial_j \beta_i) A^3_{3,i} a + F \\ 
  a(0) = 0,
 \end{array}
  \right.
$$
 with $\varepsilon_{i,j} \in \{0,1\}$ for any $1 \leq i,j \leq 2$. 
 We also assume that $\alpha$ lies in $L^2L^3$, that all the $\beta_i$'s are in $L^2H^{\frac 32}$ and that the force $F$ belongs to $L^1L^p \cap L^1L^{\frac 65}$.
 Then $a$ is actually in $L^{\infty}L^p \cap L^2W^{1,p}$.
 \label{GainRegulariteAnisotrope}
\end{lemma}

\begin{proof}[Sketch of proof]
 For simplicity, we again assume that only one coefficient $\varepsilon_{i,j}$ is nonzero and write $\partial_j \beta$ instead of $\partial_j \beta_i$.
 We abbreviate the whole right-hand side of the equation by $\tilde{F}$.
 First, we mollify the force fields $\alpha, \partial_j \beta, F$ and the weakly anisotropic operator $A^3_{3,i}$ by some regularizing kernel $\rho_{\delta}$.
 This mollified right-hand side will be denoted by $\tilde{F}^{\delta}$, even though it is not exactly equal to $\rho_{\delta} \ast \tilde{F}$.
 This regularization allows us to build smooth solutions $a_{\delta}$ to the modified equation.
 In a second step, Lemma \ref{EstimeeEnergieLp} gives us estimates in the energy space associated to $L^p$ which are uniform in $\delta$.
 These estimates write, recalling that $a_{\delta}(0) = 0$,
$$
  \frac 1p \|a_{\delta}(t)\|_{L^p}^p + (p-1) \int_0^t \||a_{\delta}(s)|^{\frac{p-2}{2}} \nabla a_{\delta}(s)\|_{L^2}^2 ds 
  = \int_0^t \int_{\mathbb{R}^3} a_{\delta}(s,x)^{p-1}\tilde{F}^{\delta}(s,x) dx ds.
$$
Repeating the computations we did for Lemma \ref{UniciteAnisotropeL6/5} and using H\"older inequality to deal with $F^{\delta}$, we get 
$$
\|a_{\delta}(t)\|_{L^p}^p \lesssim \int_0^t \|a_{\delta}(s)\|_{L^p}^p (\|\alpha^{\delta}(s)\|_{L^3}^2 + \|\beta^{\delta}(s)\|_{H^{\frac 32}}^2) ds
+ \int_0^t \|a_{\delta}(s)\|_{L^p}^{p-1} \|F^{\delta}(s)\|_{L^p} ds.
$$
We detail here how to deal with the new term added by $F^{\delta}$.
Let us denote, for $T > 0$,
$$
 M_{\delta}(T) := \sup_{0 \leq t \leq T} \|a_{\delta}(t)\|_{L^p}.
$$
For $0 \leq t \leq T$, we have
\begin{align*}
 \|a_{\delta}(t)\|_{L^p}^p & \lesssim \int_0^T \|a_{\delta}(s)\|_{L^p}^p (\|\alpha^{\delta}(s)\|_{L^3}^2 + \|\beta^{\delta}(s)\|_{H^{\frac 32}}^2) ds
+ M_{\delta}(T)^{p-1} \int_0^T \|F^{\delta}(s)\|_{L^p} ds \\
& \lesssim \int_0^T \|a_{\delta}(s)\|_{L^p}^p (\|\alpha^{\delta}(s)\|_{L^3}^2 + \|\beta^{\delta}(s)\|_{H^{\frac 32}}^2) ds
+ M_{\delta}(T)^{p-1} \|F\|_{L^1L^p}.
\end{align*}
Taking the supremum over $0 \leq t \leq T$ in the left-hand side gives
$$
  \|M_{\delta}(T)\|_{L^p}^p  \lesssim \int_0^T \|a_{\delta}(s)\|_{L^p}^p (\|\alpha^{\delta}(s)\|_{L^3}^2 + \|\beta^{\delta}(s)\|_{H^{\frac 32}}^2) ds
+ M_{\delta}(T)^{p-1} \|F\|_{L^1L^p}.
$$
Viewing the above equation as an algebraic inequality between positive numbers, we get
$$
  \|M_{\delta}(T)\|_{L^p}  \lesssim \left( \int_0^T \|a_{\delta}(s)\|_{L^p}^p (\|\alpha^{\delta}(s)\|_{L^3}^2 + \|\beta^{\delta}(s)\|_{H^{\frac 32}}^2) ds \right)^{\frac 1p}
+\|F\|_{L^1L^p}.
$$
Taking again the $p-$th power and owing to the inequality $(a+b)^p \lesssim a^p + b^p$, we have
$$
  \|M_{\delta}(T)\|_{L^p}^p  \lesssim \left( \int_0^T \|a_{\delta}(s)\|_{L^p}^p (\|\alpha^{\delta}(s)\|_{L^3}^2 + \|\beta^{\delta}(s)\|_{H^{\frac 32}}^2) ds \right)
+\|F\|_{L^1L^p}^p.
$$
Finally, since $\|a_{\delta}(T)\|_{L^p} \leq M_{\delta}(T)$ for all $T > 0$, Gr\"onwall's inequality entails that, for some constant $C > 0$,
$$
 \|a_{\delta}(T)\|_{L^p} \leq C \|F\|_{L^1L^p} \exp\left(C \int_0^T \|\alpha(s)\|_{L^3}^2 + \|\beta(s)\|_{H^{\frac 32}}^2 ds \right).
$$
Having this bound and its analogue for the exponent $p = \frac 65$, thanks to the assumptions we did on $F$, 
we get a solution of our problem in both the energy spaces associated to $L^{\frac 65}$ and $L^p$.
We conclude that this new solution is actually equal to $a$ thanks to Lemma \ref{UniciteAnisotropeL6/5}.
\end{proof}

\begin{lemma}
 Let $v$ be a fixed, divergence free vector field in $L^2H^1$.
Let $a$ be a $L^{\infty}L^{\frac 65} \cap L^2W^{1,\frac 65}$ solution of the linear system
$$
 \left \{
 \begin{array}{c c}
   \partial_t a + \nabla \cdot (a v) - \Delta a = \alpha a + \sum_{i,j = 1,2} \varepsilon_{i,j} (\partial_j \beta_i) A^3_{3,i} a + F_1 + F_2 \\ 
  a(0) = 0,
 \end{array}
  \right.
$$
 with $\varepsilon_{i,j} \in \{0,1\}$ for any $1 \leq i,j \leq 2$. 
 We assume that $\alpha$ lies in $L^2L^3$ and that all the $\beta_i$'s are in $L^2H^{\frac 32}$.
 The exterior forces $F_1$ and $F_2$ belong respectively to $L^1L^{\frac 32} \cap L^1L^{\frac 65}$ and $L^{\frac 43}L^{\frac 65} \cap L^1L^{\frac 65}$.
 Then $a$ is actually in $L^{\infty}L^{\frac 32} \cap L^2W^{1, \frac 32}$.
 \label{GainRegulariteAnisotropeBis}
\end{lemma}
\begin{proof}[Sketch of proof]
 We essentially have to repeat the proof of Lemma \ref{GainRegulariteAnisotrope}, apart from estimating the term coming from $F_2$.
 Keeping the same notations as in the last proof, we have
$$
  \int_0^t \int_{\mathbb{R}^3} a_{\delta}(s,x)^{\frac 12} F^{\delta}(s,x) dx ds
  \leq \int_0^t \|F^{\delta}(s)\|_{L^{\frac 65}} \|a_{\delta}(s)^{\frac 12}\|_{L^6} ds.
$$
Using the identity $ \|a_{\delta}(s)^{\frac 12}\|_{L^6} =  \|a_{\delta}(s)^{\frac 34}\|_{L^4}^{\frac 23}$ and the Sobolev embedding $H^{\frac 34} \hookrightarrow L^4$, we get
$$
  \int_0^t \|F^{\delta}(s)\|_{L^{\frac 65}} \|a_{\delta}(s)^{\frac 12}\|_{L^6} ds
  \lesssim \int_0^t \|F^{\delta}(s)\|_{L^{\frac 65}} \|a_{\delta}(s)^{\frac 34}\|_{L^2}^{\frac 16} \||a(s)|^{-\frac 14} \nabla a(s)\|_{L^2}^{\frac 12} ds.
$$
 Now, Young inequality gives us, for some constant $C > 0$,
  \begin{multline*}
 \int_0^t \|F^{\delta}(s)\|_{L^{\frac 65}} \|a_{\delta}(s)^{\frac 34}\|_{L^2}^{\frac 16} \||a(s)|^{-\frac 14} \nabla a(s)\|_{L^2}^{\frac 12} ds
 \leq \frac{1}{10} \int_0^t \||a(s)|^{-\frac 14} \nabla a(s)\|_{L^2}^2 ds \\
 + \int_0^t \|a_{\delta}(s)^{\frac 34}\|_{L^2}^2 \|F^{\delta}(s)\|_{L^{\frac 65}}^{\frac 43} ds 
 + C \int_0^t \|F^{\delta}(s)\|_{L^{\frac 65}}^{\frac 43} ds.
 \end{multline*}
 Plugging this finaly bound in the energy estimate performed in $L^{\frac 32}$, the rest of the proof is the same as for Lemma \ref{GainRegulariteAnisotrope}.
\end{proof}

\begin{lemma}
Let $v$ be a fixed, divergence free vector field in $L^2H^1$.
  Let $A$ be a matrix-valued function in $L^2L^3$.
 Let $K$ be a matrix whose coefficients are homogeneous, isotropic Fourier multipliers of order $0$.
 Let $F_1$ be a fixed function in $L^1L^{\frac 32} \cap L^1L^{\frac 65}$ and $F_2$ be fixed in $L^{\frac 43}L^{\frac 65} \cap L^1L^{\frac 65}$.
 Let $a$ be a solution in $(L^2L^2)^2$ of the equation
 
$$
  \left\{ 
  \begin{array}{c c}
   \partial_t a + \nabla \cdot (a \otimes v) - \Delta a = A K a + F_1 + F_2  \\
   a(0) = 0.
  \end{array}
  \right.
$$
Then $a$ is actually in $L^{\infty} L^{\frac 32} \cap L^2W^{1, \frac 32}$.
\label{GainRegulariteSystemeBis}
\end{lemma}

\begin{proof}
 This lemma essentially combines the proofs of Lemmas \ref{GainRegulariteLp}, \ref{GainRegulariteAnisotrope} and \ref{GainRegulariteAnisotropeBis}, so we shall not repeat them.
\end{proof}

\begin{lemma}
 Let $v_0$ be a divergence free vector field in $L^{\frac 32} \cap L^2$. Then any Leray solution of the Navier-Stokes system
$$
 \left \{
 \begin{array}{c c}
   \partial_t v + \nabla \cdot (v \otimes v) - \Delta v = - \nabla p \\
   \text{div } v = 0 \\
  v(0) = v_0
 \end{array}
  \right.
$$
 belongs, in addition to the classical energy space $L^{\infty}L^2 \cap L^2H^1$, to $L^{\infty}L^{\frac 32} \cap L^2W^{1,\frac 32}$.
\end{lemma}

\begin{proof}
 Let $v$ be a Leray solution of the Navier-Stokes system, which exists by classical approximation arguments. 
 Then, letting $F := - \mathbb{P} \nabla \cdot (v \otimes v) = - \mathbb{P} (v \cdot \nabla v)$ where $\mathbb{P}$ denotes the Leray projection on divergence free vector fields, $v$ solves the heat equation
$$
 \left \{
 \begin{array}{c c}
   \partial_t v - \Delta v = F \\
  v(0) = v_0.
 \end{array}
  \right.
$$
 That $F$ belongs to $L^1L^{\frac 32}$ is easily obtained by the continuity of $\mathbb{P}$ on $L^{\frac 32}$.
The result follows from an energy estimate in $L^{\frac 32}$.
\label{IntegrabiliteSolutionLeray}
\end{proof}

\section{Case of the torus}

Let us now state the first main theorem of this paper.

\begin{theorem}
 Let $u$ be a Leray solution of the Navier-Stokes equations set in $\mathbb{R}_+ \times \mathbb{T}^3$ 
$$
  \left\{ 
  \begin{array}{c c}
   \partial_t u + \nabla \cdot (u \otimes u) - \Delta u = - \nabla p \\
   u(0) = u_0
  \end{array}
  \right.
$$
with initial data $u_0$ in $L^2(\mathbb{T}^3)$.
Assume the existence of a time interval $]T_1,T_2[$ such that its third component $u^3$ satisfies
$$
  u^3 \in L^2(]T_1,T_2[, W^{2,\frac 32}(\mathbb{T}^3)).
$$
Then $u$ is actually smooth in time and space on $]T_1,T_2[ \times \mathbb{T}^3$ and satisfies the Navier-Stokes equations in the classical, strong sense.
\end{theorem}

Let $\chi, \varphi$ be smooths cutoffs in time, localised inside $]T_1,T_2[$.
Let $\omega$ be the third component of $\Omega := \text{rot } v$.
Denote $\chi \omega$ by $\omega'$.
The equation satisfied by $\omega'$ writes
$$
 \partial_t  \omega' + \nabla \cdot (\omega' u) - \Delta \omega' = \chi \Omega \cdot \nabla u^3 + \omega \partial_t \chi.
$$

Denote $F := \chi \Omega \cdot \nabla u^3 + \omega \partial_t \chi$. 
As $u$ is a Leray solution of the Navier-Stokes equations, we know that $\Omega$ belongs to $L^2L^2$.
Thus, $\omega'$ also lies in $L^2L^2$.
On the other hand, the assumption made on $u^3$ tells us in particular that $\Omega \cdot \nabla u^3$ belongs to $L^1L^{\frac 65}$.
That $\omega \partial_t \chi$ also belongs to $L^1L^{\frac 65}$ follows directly from the compactness of $\mathbb{T}^3$.

We are now in position to apply Lemma \ref{GainRegulariteLp}, which tells us that $\omega'$ is actually in $L^{\infty}L^{\frac 65} \cap L^2W^{1,\frac 65}$.
Let us now expand the quantity $\Omega \cdot \nabla u^3$ in terms of $\omega$ and $u^3$. We have, after some simplifications,
$$
  \Omega \cdot \nabla u^3 = \partial_3 u^3 \omega  + \partial_2 u^3 \partial_3 u^1 - \partial_1 u^3 \partial_3 u^2.
$$
Performing a div-curl decomposition of $u^1$ and $u^2$ in terms of $\partial_3 u^3$ and $\omega$, we have
\begin{align*}
  \Omega \cdot \nabla u^3 & =  \partial_3 u^3 \omega  + \partial_2 u^3 (- A_{1,3}^3 \partial_3 u^3 - A_{2,3}^3 \omega) - \partial_1 u^3 (- A_{2,3}^3 \partial_3 u^3 + A_{1,3}^3 \omega) \\
  & = \partial_3 u^3 \omega + \mathcal{A}(\omega, u^3) + \mathcal{B}(u^3, u^3),
\end{align*}
where we defined as shorthands the operators
\begin{align*}
 \mathcal{A}(\omega, u^3) & := - \partial_2 u^3  A_{2,3}^3 \omega - \partial_1 u^3  A_{1,3}^3 \omega \\
 \mathcal{B}(u^3, u^3)  & := - \partial_2 u^3 A_{1,3}^3 \partial_3 u^3 + \partial_1 u^3  A_{2,3}^3 \partial_3 u^3.
\end{align*}
Notice that the div-curl decomposition forces the appearance of weakly anisotropic operators acting either on $\omega$ or $u^3$.
Assume from now on that the condition
$$
 \text{supp }\chi \subset \{\varphi \equiv 1 \}.
$$
holds.
Under this condition, the equation on $\omega'$ then reads
\begin{align*}
 \partial_t \omega' + \nabla \cdot (\omega' u) - \Delta \omega' & = 
 \chi \omega \partial_3 u^3 + \chi \mathcal{A}(\omega, u^3) + \chi \mathcal{B}(u^3, u^3) + \omega \partial_t \chi \\
 & = \omega' \partial_3 u^3 + \mathcal{A}(\omega', \varphi u^3) + \mathcal{B}(\chi u^3, \varphi u^3) + \omega \partial_t \chi,
\end{align*}
because the cutoffs $\chi$ and $\varphi$ act only on time.

It follows from the assumptions on $u^3$ that $\mathcal{B}(\chi u^3, \varphi u^3)$ belongs to $L^1L^{\frac 32}$.
Moreover, $\omega \partial_t \chi$ also belongs to $L^1L^{\frac 32}$.

By Lemma \ref{GainRegulariteAnisotrope}, $\omega'$ is actually in $L^{\infty}L^{\frac 32} \cap L^2W^{1, \frac 32}$.

Let us now write the system of equations satisfied by the other components of the vorticity, which we respectively denote by $\omega_1$ and $\omega_2$.
We have 

$$
 \left\{
 \begin{array}{c c}
  \partial_t \omega_1 + \nabla \cdot (\omega_1 u) - \Delta \omega_1 = \partial_3 u^1 \partial_1 u^2 - \partial_2 u^1 \partial_1 u^3 \\
  \partial_t \omega_2 + \nabla \cdot (\omega_2 u) - \Delta \omega_2 = \partial_1 u^2 \partial_2 u^3 - \partial_3 u^2 \partial_2 u^1 .
 \end{array}
\right.
$$
We now perform a div-curl decomposition of $u^1$ with respect to the second variable. That is, we write that
$$
 u^1 = \partial_3 \Delta_{(1,3)}^{-1} \omega_2 - \partial_1 \Delta_{(1,3)}^{-1} \partial_2 u^2.
$$
In turn, we have
\begin{align*}
 \partial_3 u^1 & = \partial_3^2 \Delta_{(1,3)}^{-1} \omega_2 - \partial_3 \partial_1 \Delta_{(1,3)}^{-1} \partial_2 u^2 \\
 & = A_{3,3}^2 \omega_2 - A_{1,3}^2 \partial_2 u^2.
\end{align*}
What we wish to emphasize is that $\partial_3 u^1$ may be expressed as an order zero \emph{isotropic} Fourier multiplier applied to $\omega_2$ and $\partial_2 u^2$.
The same reasoning applies to $\partial_3 u^2$, which may decomposed in terms of $\omega_1$ et $\partial_1 u^1$.
The fact that there is no (weakly) anisotropic operator here is a great simplification compared to the study of $\omega_3$, for which such a complication was unavoidable.
The system on $(\omega_1, \omega_2)$ may be recast in the following form : 
$$
 \left\{
 \begin{array}{c c}
  \partial_t \omega_1 + \nabla \cdot (\omega_1 u) - \Delta \omega_1 
  = (A_{3,3}^2 \omega_2 - A_{1,3}^2 \partial_2 u^2) \partial_1 u^2 - \partial_2 u^1 \partial_1 u^3 \\
  \partial_t \omega_2 + \nabla \cdot (\omega_2 u) - \Delta \omega_2 
  = \partial_1 u^2 \partial_2 u^3 + (A_{3,3}^1\omega_1 + A_{2,3}^1\partial_1 u^1) \partial_2 u^1 .
 \end{array}
\right.
$$
Informally, the above system behaves roughly like its simplified version
$$
 \left\{
 \begin{array}{c c}
  \partial_t \omega_1 + \nabla \cdot (\omega_1 u) - \Delta \omega_1 = (\omega_2 - \partial_2 u^2) \partial_1 u^2 - \partial_2 u^1 \partial_1 u^3 \\
  \partial_t \omega_2 + \nabla \cdot (\omega_2 u) - \Delta \omega_2 = \partial_1 u^2 \partial_2 u^3 + (\omega_1 + \partial_1 u^1) \partial_2 u^1,
 \end{array}
\right.
$$
which is much simpler to understand and shall make the upcoming computations clearer.
Let us denote, as we did for $\omega = \omega_3$, $\omega_1' := \chi \omega_1$ and $\omega_2' := \chi \omega_2$.
Applying the time cutoff $\chi$ to the system on $(\omega_1, \omega_2)$, we get
$$
 \left\{
 \begin{array}{c c}
  \partial_t \omega_1' + \nabla \cdot (\omega_1' u) - \Delta \omega_1' 
  = \varphi \partial_1 u^2 A_{3,3}^2 \omega_2' - \varphi \partial_1 u^2 A_{1,3}^2 (\chi \partial_2 u^2) - (\chi \partial_2 u^1) (\varphi \partial_1 u^3) + \omega_1 \partial_t \chi \\
  \partial_t \omega_2' + \nabla \cdot (\omega_2' u) - \Delta \omega_2' 
  = \varphi \partial_2 u^1 A_{3,3}^1\omega_1' + \varphi \partial_2 u^1 A_{2,3}^1(\chi\partial_1 u^1) +  (\chi \partial_1 u^2) (\varphi \partial_2 u^3) + \omega_2 \partial_t \chi.
 \end{array}
\right.
$$

Finally, applying the same decomposition to $u^1$ and $u^2$, we have four equations of the type
$$
 \partial_1 u^1 = - A_{1,1}^3 \omega_3 - A_{1,2}^3 \partial_3 u^3,
$$
which allow us to control, for $1 \leq i,j \leq 2$, $\partial_i u^j$ in $L^{\infty}L^{\frac 32} \cap L^2W^{1, \frac 65}$ in terms of $\omega_3$ and $\partial_3 u^3$ in the same space.
Thus, what we have gained through the regularity enhancement on $\omega_3$ is the control of four components of the jacobian of $u$, in addition to the three provided by the assumption on $u^3$.
For this reason, the system we have on $(\omega_1, \omega_2)$ may be viewed as an affine and isotropic one with all exterior forces in scaling invariant spaces.
For instance, $\varphi \partial_2 u^1$ belongs to $L^2L^3$, while the exterior forces lie in $L^1L^{\frac 32}$.
Lemma \ref{GainRegulariteSysteme} now implies that both $\omega_1'$ and $\omega_2'$ are in $L^{\infty} L^{\frac 32} \cap L^2 W^{1, \frac 32}$.

We now have proven that the whole vorticity $\Omega$ belongs to $L^4L^2$ by Sobolev embeddings.
In turn, it implies that the whole velocity field belongs to $L^4H^1$. 
The main theorem then follows from the application of the usual Serrin criterion.

\section{Local case in $\mathbb{R}^3$.}

We state the second main theorem of this paper.

\begin{theorem}
 Let $u$ be a Leray solution of the Navier-Stokes equations set in $\mathbb{R}_+ \times \mathbb{R}^3$ 
$$
  \left\{ 
  \begin{array}{c c}
   \partial_t u + \nabla \cdot (u \otimes u) - \Delta u = - \nabla p \\
   u(0) = u_0
  \end{array}
  \right.
$$
 with initial data $u_0$ in $L^2(\mathbb{R}^3) \cap L^{\frac 32}(\mathbb{R}^3)$.
 Assume the existence of a time interval $]T_1,T_2[$ and a spatial domain $D \Subset \mathbb{R}^3$ of compact closure such that its third component $u^3$ satisfies
$$
  u^3 \in L^2(]T_1,T_2[, W^{2,\frac 32}(D)).
$$
Then, on $]T_1,T_2[ \times D$, $u$ is actually smooth in time and space and satisfies the Navier-Stokes equations in the classical, strong sense.
\end{theorem}

Let us describe in a few words our strategy for this case.
Compared to the torus, there are two main differences to notice.
First, since the assumption on $u^3$ was made on the whole space, the cutoffs acted only in time. 
The difference between the original Navier-Stokes equation and its truncated version was thus only visible in one term, rendering our strategy easier to apply.
On the other hand, since the torus has finite measure, the Lebesgue spaces form a decreasing family of spaces.
This fact allowed us to lose some integrability when we wanted to embed different forcing terms in the same space.
This last difference will become visible when dealing with commutators between Fourier multipliers and the cutoff functions, thus lengthening a little bit the proof, compared to the torus case.
For that technical reason, we added an assumption on the initial data which was trivially true in the torus case, thanks to the aforementioned embedding of Lebesgue spaces.

Let $\chi, \varphi$ be smooths cutoffs in space and time, localised inside $]T_1,T_2[ \times D$.
Let $\omega$ be the third component of $\Omega := \text{rot } v$.
Denote $\chi \omega$ by $\omega'$.
The equation satisfied by $\omega'$ writes
$$
 \partial_t  \omega' + \nabla \cdot (\omega' u) - \Delta \omega' = \chi \Omega \cdot \nabla u^3 + \mathcal{C}(\omega, \chi),
$$
where $\mathcal{C}(\omega,\chi)$ stands for all the cutoff terms. Namely, we have
$$
 \mathcal{C}(\omega, \chi) := \omega \partial_t \chi + \omega u \cdot \nabla \chi - \omega \Delta \chi - 2 \nabla \omega \cdot \nabla \chi.
$$
As $\chi$ is smooth and has compact support, we claim that $\mathcal{C}(\omega, \chi)$ belongs to $L^1L^{\frac 32} + L^2H^{-1}$.
Because $\chi$ has compact support in space, the terms in $L^1L^{\frac 32}$ also lie in $L^1L^{\frac 65}$.
Finally, the quantity $\chi \Omega \nabla u^3$ clearly belongs to $L^1L^{\frac 65}$.
Let now $\omega'_{(1)}$ be the unique solution in $L^{\infty}L^{\frac 65} \cap L^2W^{1,\frac 65}$ of the equation 
$$
 \partial_t  \omega_{(1)}' + \nabla \cdot (\omega_{(1)}' u) - \Delta \omega_{(1)}' = \chi \Omega \cdot \nabla u^3 + \omega \partial_t \chi + \omega u \cdot \nabla \chi - \omega \Delta \chi
$$
with the initial condition $\omega_{(1)}'(0) = 0$, which exists thanks to Lemma \ref{EstimeeEnergieLp} and is unique thanks to Lemma \ref{UniciteL2L2}.
Similarly, let $\omega_{(2)}'$ be the unique solution in $L^{\infty}L^2 \cap L^2H^1$ of 
$$
 \partial_t  \omega_{(2)}' + \nabla \cdot (\omega_{(2)}' u) - \Delta \omega_{(2)}' = - 2 \nabla \omega \cdot \nabla \chi.
$$
with the initial condition $\omega_{(2)}'(0) = 0$.
Let $\tilde{\omega}' := \omega'_{(1)} + \omega'_{(2)} - \omega'$. 
From the regularity we have on each term, $\tilde{\omega}'$ belongs to $L^2_{loc}L^2$ and satisfies
$$
\partial_t  \tilde{\omega}' + \nabla \cdot (\tilde{\omega}' u) - \Delta \tilde{\omega}' = 0
$$
along with the initial condition $\tilde{\omega}'(0) = 0$.
Lemma \ref{UniciteL2L2} then implies that $\tilde{\omega}' \equiv 0$, from which it follows that
$$
 \omega' = \omega_{(1)}' + \omega_{(2)}'
$$
By local embeddings of Lebesgue spaces, $\omega_{(2)}'$ also belongs to $L^{\infty}L^{\frac 65}_{loc} \cap L^2W^{1,\frac 65}_{loc}$.
On the other hand, it is rather trivial that $\omega_{(1)}'$ also belongs to $L^{\infty}L^{\frac 65}_{loc} \cap L^2W^{1,\frac 65}_{loc}$.
Now, since $\omega'$ has compact support in space, it follows that $\omega'$ belongs to the full space $L^{\infty}L^{\frac 65} \cap L^2W^{1,\frac 65}$.
In particular, the forcing term $\nabla \omega \cdot \nabla \chi$ is now an integrable vector field, instead of a mere $L^2H^{-1}$ distribution.
At this stage, because the reasoning is valid for \textit{any} cutoff $\chi$ supported in $]T_1,T_2[ \times D$, we have proved that the third component $\omega$ of the vorticity of $u$ has the regularity
$$
 \omega \in L^{\infty}_{loc}(]T_1, T_2[, L^{\frac 65}_{loc}(D)) \cap L^2_{loc}(]T_1,T_2[, W^{1,\frac 65}_{loc}(D)).
$$
In particular, such a statement allows us to improve the regularity of $\mathcal{C}(\omega, \chi)$ to $L^1L^{\frac 32} + L^2L^{\frac 65}$.
Such a gain will be of utmost importance near the end of the proof.
Expanding again the product $\Omega \cdot \nabla u^3$ in terms of $\omega$ and $u^3$ only, we have
$$
 \partial_t \omega' + \nabla \cdot (\omega' u) - \Delta \omega'  = 
 \chi \omega \partial_3 u^3 + \chi \mathcal{A}(\omega, u^3) + \chi \mathcal{B}(u^3, u^3) + \mathcal{C}(\omega, \chi).
$$
From now on, we enforce the condition
$$
 \text{supp }\chi \subset \{\varphi \equiv 1 \}.
$$
Now, because the cutoff $\chi$ acts both in space and time, we have to carefully compute the associated commutators with the operators $\mathcal{A}$ and $\mathcal{B}$.
First, let us notice that $\mathcal{A}$ is local in its variable $u^3$, which allows us to write that 
$$
 \chi \mathcal{A}(\omega, u^3) = \chi \mathcal{A}(\omega, \varphi u^3).
$$
On the other hand, for $i = 1,2$, 
\begin{align*}
 \chi A^3_{i,3} \omega & = \chi \partial_i \Delta_{(1,2)}^{-1} (\partial_3 \omega) \\
 & = [\chi, \partial_i \Delta_{(1,2)}^{-1}](\partial_3 \omega) + \partial_i \Delta_{(1,2)}^{-1}(\chi \partial_3 \omega) \\
 & = [\chi, \partial_i \Delta_{(1,2)}^{-1}](\partial_3 \omega) + A^3_{i,3}(\chi \omega) - \partial_i \Delta_{(1,2)}^{-1} (\omega \partial_3 \chi)
\end{align*}
We now estimate the two remainder terms in $L^1L^{\frac 32}$. By Sobolev embeddings in $\mathbb{R}^2$, we have, for $t > 0$ and $x_3 \in \mathbb{R}$,
$$
 \|\left(\partial_i \Delta_{(1,2)}^{-1} (\omega \partial_3 \chi) \right)(t,\cdot,x_3)\|_{L^6(\mathbb{R}^2)} 
 \lesssim \|(\omega \partial_3 \chi) (t,\cdot,x_3)\|_{L^{\frac 32}(\mathbb{R}^2)}.
$$
Thus,
$$
  \|\partial_i \Delta_{(1,2)}^{-1} (\omega \partial_3 \chi)\|_{L^2L^{\frac 32}L^6} \lesssim \|\omega \partial_3 \chi\|_{L^2L^{\frac 32}} \lesssim \|\omega\|_{L^2L^2} \|\nabla \chi\|_{L^{\infty}L^6}.
$$
The commutator is a little bit trickier. First, we write
$$
 \partial_3 \omega = \partial_3 (\partial_1 u^2 - \partial_2 u^1) = \partial_1 (\partial_3 u^2) - \partial_2 (\partial_3 u^1).
$$
In order to continue the proof, we need a commutator lemma, which we state and prove below for the sake of completeness, despite its ordinary nature.
\begin{lemma}
 Let $f$ be in $L^{\frac 32}(\mathbb{R}^2)$ and $\chi$ be a test function. The following commutator estimates hold :
$$
  \|[\chi, \nabla \Delta^{-1}](\nabla f)\|_{L^6(\mathbb{R}^2)} \lesssim \|\nabla \chi\|_{L^{\infty}(\mathbb{R}^2)} \|f\|_{L^{\frac 32}(\mathbb{R}^2)}
$$
 and
$$
  \|[\chi, \nabla^2 \Delta^{-1}](f)\|_{L^6(\mathbb{R}^2)} \lesssim \|\nabla \chi\|_{L^{\infty}(\mathbb{R}^2)} \|f\|_{L^{\frac 32}(\mathbb{R}^2)}.
$$
 \label{Commutateur}
\end{lemma}

\begin{proof}
We notice that the first estimate may be deduced from the second thanks to the identity
$$[\chi, \nabla \Delta^{-1}](\nabla f) = [\chi, \nabla^2 \Delta^{-1}](f) + \nabla\Delta^{-1} (f \nabla \chi). $$
Since the operator $\nabla \Delta^{-1}$ is continous from $L^{\frac 32}(\mathbb{R}^2)$ to $L^6(\mathbb{R}^2)$, we get
$$\|\nabla\Delta^{-1} (f \nabla \chi)\|_{L^6(\mathbb{R}^2)} \lesssim \|f\nabla \chi\|_{L^{\frac 32}(\mathbb{R}^2)} \lesssim
\|f\|_{L^{\frac 32}(\mathbb{R}^2)} \|\nabla \chi\|_{L^{\infty}(\mathbb{R}^2)}$$
It only remains to study the second commutator, which we denote by $C_{\chi}$.
There exist numerical constants $c_1,c_2$ such that, for almost every $x \in \mathbb{R}^2$,
$$C_{\chi}(x) = \int_{\mathbb{R}^2} \left(c_1 \frac{(x-y)\otimes (x-y)}{|x-y|^4} +  \frac{c_2}{|x-y|^2}I_2 \right) (\chi(x) - \chi(y)) f(y) dy.$$
This yields
$$|C_{\chi}(x)| \lesssim \|\nabla \chi\|_{L^{\infty}(\mathbb{R}^2)} \int_{\mathbb{R}^2} \frac{|f(y)|}{|x-y|} dy = \|\nabla \chi\|_{L^{\infty}(\mathbb{R}^2)} (|f| \ast |\cdot|^{-1})(x). $$
Applying the Hardy-Littlewood-Sobolev inequality to $f$, we get
$$\|C_{\chi}\|_{L^6(\mathbb{R}^2)} \lesssim \|\nabla \chi\|_{L^{\infty}(\mathbb{R}^2)} \|f\|_{L^{\frac 32}(\mathbb{R}^2)} $$
as we wanted.
\end{proof}
Thanks to Lemma \ref{Commutateur}, we have the estimate
$$
 \| [\chi, \partial_i \Delta_{(1,2)}^{-1}](\partial_1 (\partial_3 u^2) ) \|_{L^6(\mathbb{R}^2)} \lesssim \|\nabla \chi\|_{L^{\infty}} \|\partial_3 u^2\|_{L^{\frac 32}(\mathbb{R}^2)},
$$
which translates into
$$
 \| [\chi, \partial_i \Delta_{(1,2)}^{-1}](\partial_1 (\partial_3 u^2) ) \|_{L^2L^{\frac 32}L^6} \lesssim \|\nabla \chi\|_{L^{\infty}} \|\partial_3 u^2\|_{L^2L^{\frac 32}}.
$$
From Lemma \ref{IntegrabiliteSolutionLeray} applied to $u$, we deduce that, in particular, $\partial_3 u^2$ belongs to $L^2L^{\frac 32}$.
Moreover, we may bound $\|\partial_3 u^2\|_{L^2L^{\frac 32}}$ by a quantity depending only on $u_0$ through its $L^2$ and $L^{\frac 32}$ norms. 
Gathering these estimates, we may write
$$
 \chi \mathcal{A}(\omega, \varphi u^3) = \mathcal{A}(\chi \omega, \varphi u^3) + \mathcal{R}(\mathcal{A}),
$$
with the remainder $\mathcal{R}(\mathcal{A})$ bounded in $L^1L^{\frac 32}$ only in terms of the initial data $u_0$, the cutoff $\chi$ and $u^3$.
In particular, it may be regarded as an exterior force independant of $\omega'$ in the sequel and scaling invariant.
The same reasoning applies to $\mathcal{B}$ : we have
$$
 \chi \mathcal{B}(u^3, \varphi u^3) = \mathcal{B}(\chi u^3, \varphi u^3) + \mathcal{R}(\mathcal{B}),
$$
with $\mathcal{R}(\mathcal{B})$ bounded in $L^1L^{\frac 32}$ only in terms of $\chi$ and $u^3$.
Finally, the equation on $\omega'$ has been rewritten as
$$
\partial_t \omega' + \nabla \cdot (\omega' u) - \Delta \omega'  = 
 \omega' \partial_3 u^3 +  \mathcal{A}(\omega', \varphi u^3) +  \mathcal{B}(\chi u^3, \varphi u^3) + \mathcal{C}(\omega, \chi) + \mathcal{R}(\mathcal{A}) + \mathcal{R}(\mathcal{B}).
$$
Applying Lemma \ref{GainRegulariteAnisotropeBis}, we deduce that the truncated vorticity $\omega'$ is actually in $L^{\infty}L^{\frac 32} \cap L^2W^{1, \frac 32}$.
Again, thanks to the div-curl decomposition, it follows that space-time truncations of $\partial_i u^j$ are controlled in the same space in terms of $\omega'$ and $u^3$, for $1 \leq i,j \leq 2$.
We now turn to the other components of the vorticity, namely $\omega_1$ and $\omega_2$. Truncating the equations and using the div-curl decomposition, we have
$$
 \left\{
 \begin{array}{c c}
  \partial_t \omega_1' + \nabla \cdot (\omega_1' u) - \Delta \omega_1' 
  = \chi (A_{3,3}^2 \omega_2 - A_{1,3}^2 \partial_2 u^2) \partial_1 u^2 - \chi \partial_2 u^1 \partial_1 u^3 + \mathcal{C}(\omega_1, \chi) \\
  \partial_t \omega_2' + \nabla \cdot (\omega_2' u) - \Delta \omega_2' 
  = \chi \partial_1 u^2 \partial_2 u^3 + \chi (A_{3,3}^1\omega_1 + A_{2,3}^1\partial_1 u^1) \partial_2 u^1 + \mathcal{C}(\omega_2, \chi).
 \end{array}
\right.
$$
Let us now write and estimate the necessary commutators.
By Lemma \ref{Commutateur}, we have, when $k$ is neither $i$ nor $j$,
$$
 \|[\chi, A^k_{i,j}](\omega_2)\|_{L^6(\mathbb{R}^2)} \lesssim \| \nabla \chi \|_{L^{\infty}} \|\omega_2\|_{L^{\frac 32}(\mathbb{R}^2)}.
$$
Thus,
$$
 \|[\chi, A^k_{i,j}](\omega_2)\|_{L^2L^{\frac 32}L^6} \lesssim \|\nabla \chi\|_{L^{\infty}} \|\omega_2\|_{L^2L^{\frac 32}}.
$$
On the other hand, by a trace theorem, we have, for $a$ in $W^{1, \frac 32}(\mathbb{R}^3)$,
$$
\|a\|_{L^{\infty}(\mathbb{R}, L^2(\mathbb{R}^2))} \lesssim \|a\|_{W^{1,\frac 32}(\mathbb{R}^3)}.
$$
These two estimates together entail that, for $1 \leq i,j \leq 2$,
$$
 \|\partial_i (\varphi u^j) [\chi, A^k_{i,j}](\omega_2)\|_{L^1L^{\frac 32}} \lesssim \|\nabla \chi\|_{L^{\infty}} \|\omega_2\|_{L^2L^{\frac 32}} \|\partial_i (\varphi u^j)\|_{W^{1,\frac 32}}.
$$
The system on $(\omega_1', \omega_2')$ may be recast as
$$
 \left\{
 \begin{array}{c c}
  \partial_t \omega_1' + \nabla \cdot (\omega_1' u) - \Delta \omega_1' 
  =  (A_{3,3}^2 \omega_2' - A_{1,3}^2  \partial_2 (\chi u^2)) \partial_1 (\varphi u^2) - \partial_2 (\chi u^1) \partial_1 (\varphi u^3) + \mathcal{C}(\omega_1, \chi) + \mathcal{R}^1 \\
  \partial_t \omega_2' + \nabla \cdot (\omega_2' u) - \Delta \omega_2' 
  =  \partial_1 (\chi u^2) \partial_2 (\varphi u^3) +  (A_{3,3}^1 \omega_1' + A_{2,3}^1  \partial_1 (\chi u^1)) \partial_2 (\varphi u^1) + \mathcal{C}(\omega_2, \chi) + \mathcal{R}^2,
 \end{array}
\right.
$$
where the remainders $\mathcal{R}^{1,2}$ contain, among other terms, the commutators we just estimated. 
The important fact is the boundedness of $\mathcal{R}^{1,2}$ in $L^1L^{\frac 32}$.
Because $\chi$ has compact support in time, the term $-2 \nabla \omega \cdot \nabla \chi$ is in $L^{\frac 43}L^{\frac 65}$.
Applying Lemma \ref{GainRegulariteSystemeBis}, it follows that both $\omega_1'$ and $\omega_2'$ belong to $L^{\infty}L^{\frac 32} \cap L^2W^{1,\frac 32}$.
The conclusion of the theorem now follows from the standard Serrin criterion.

\end{document}